\documentclass[12pt]{article}
\font\smallit=cmti10

\usepackage{amssymb,latexsym,amsmath,epsfig,amsthm}
\usepackage{rotating}
\usepackage{graphicx}
\usepackage{amssymb}
\usepackage{lineno}
\usepackage{enumitem}
\usepackage{cite}
\usepackage[top=3cm, bottom=3cm, left=3.5cm, right=3.5cm]{geometry}
\makeatletter
\renewcommand\section{\@startsection {section}{1}{\z@}
	{-30pt \@plus -1ex \@minus -.2ex}
	{2.3ex \@plus.2ex}
	{\normalfont\normalsize\bfseries\boldmath}}

\renewcommand\subsection{\@startsection{subsection}{2}{\z@}
	{-3.25ex\@plus -1ex \@minus -.2ex}
	{1.5ex \@plus .2ex}
	{\normalfont\normalsize\bfseries\boldmath}}
\renewcommand{\@seccntformat}[1]{\csname the#1\endcsname. }
\makeatother
\newtheorem{theorem}{Theorem}
\newtheorem{lemma}{Lemma}

\newtheorem{corollary}{Corollary}
\theoremstyle{definition}

\begin{document}
	
	\begin{center}
		\textbf{\large A note on distinct differences in $t$-intersecting families} 
		\vskip 20pt
		\textbf{Jagannath Bhanja and Sayan Goswami}\\
		{\smallit The Institute of Mathematical Sciences, A CI of Homi Bhabha National Institute, C.I.T. Campus, Taramani, Chennai 600113, India}\\
		{\tt jbhanja@imsc.res.in, sayangoswami@imsc.res.in}
	\end{center}
	\vskip 20pt

	\begin{abstract}
	 For a family $\mathcal{F}$ of subsets of $\{1,2,\ldots,n\}$, let $\mathcal{D}(\mathcal{F}) = \{F\setminus G: F, G \in \mathcal{F}\}$ be the collection of all (setwise) differences of $\mathcal{F}$. The family $\mathcal{F}$ is called a $t$-intersecting family, if for some positive integer $t$ and any two members $F, G \in \mathcal{F}$ we have $|F\cap G| \geq t$. The family $\mathcal{F}$ is simply called intersecting if $t=1$. Recently, Frankl proved an upper bound on the size of  $\mathcal{D}(\mathcal{F})$ for the intersecting families $\mathcal{F}$. In this note we extend the result of Frankl to $t$-intersecting families.  
	\end{abstract}
	
	\vspace{10pt}
	
	
	\vspace{10pt}

	\section{Introduction}
	We denote the standard $n$-element set $\{1,2,\ldots,n\}$ by $[n]$, the set of all subsets of $[n]$ by $2^{[n]}$, and $\binom{[n]}{k}$ by the collection of all $k$-element subsets of $[n]$ for $0 \leq k \leq n$. We also use the standard notation $|S|$ for the cardinality of a set $S$ and $\lfloor N \rfloor$ for the largest integer less than or equal to $N$.   
	
	A family $\mathcal{F}$ of subsets of $[n]$ is said to be a $t$-intersecting family for some positive integer $t$, if $|F\cap G| \geq t$ for any two members $F, G \in \mathcal{F}$. A $1$-intersecting family is simply called an intersecting family. The Erd\H{o}s-Ko-Rado theorem \cite{erdos} gives the tight upper bound on the size of $t$-intersecting families. 
    
    \begin{theorem} [Erd\H{o}s-Ko-Rado theorem for $t$-intersecting family] \label{ekr-t-theorem} There exists some $n_0(k, t)$ such that if $n \geq n_0(k, t)$ and $\mathcal{F} \subset \binom{[n]}{k}$ is $t$-intersecting, then  
    	\[|\mathcal{F}| \leq \binom{n-t}{k-t}.\]
    \end{theorem}

    The upper bound is tight for $n_0(k, t)=(t+1)(k-t+1)$; this was proved by Frankl \cite{frankl78} for $t\geq 15$ and for all $t$ using a different technique by Wilson \cite{wilson}. 
    
    For a family $\mathcal{F}$, let 
    \[\mathcal{D}(\mathcal{F}) = \{F\setminus G: F, G \in \mathcal{F}\}\]
    be the collection of all (setwise) differences of $\mathcal{F}$. Here we allow the empty set $\emptyset$ to be in $\mathcal{F}$ when $\mathcal{F} \neq \emptyset$. In this direction Marica and Sch\"{o}nheim \cite{marica} proved the following theorem. 
    
    \begin{theorem} [Marica-Sch\"{o}nheim theorem] \label{ms-theorem} For a nonempty family $\mathcal{F} \subset 2^{[n]}$ one has  
    	\[|\mathcal{D}(\mathcal{F})| \geq |\mathcal{F}|.\]
    \end{theorem}
    
    Recently, Frankl \cite{frankl21} proved the following upper bound for $\mathcal{D}(\mathcal{F})$ when $\mathcal{F}$ is an intersecting family. 
    
    \begin{theorem} [Frankl] \label{frankl-theorem} Suppose that $\mathcal{F} \subset \binom{[n]}{k}$ is an intersecting family with $n \geq k(k+3)$. Then    
    	\[|\mathcal{D}(\mathcal{F})| \leq \binom{n-1}{k-1}+\binom{n-1}{k-2}+\cdots+\binom{n-1}{0}.\]
    \end{theorem}
    
    For further developments in this direction see \cite{frankl22}. 
    
    In this note we extend Theorem \ref{frankl-theorem} to the $t$-intersecting families $\mathcal{F}$ of $\binom{[n]}{k}$. More precisely we prove the following theorem. 
    
    \begin{theorem} \label{main-theorem} Suppose that $\mathcal{F} \subset \binom{[n]}{k}$ is a $t$-intersecting family with $n \geq (k-t)(k+t+3)\cdot \lfloor\left( \frac{k+t+3}{k+t+1}\right)^{t-1}\rfloor$. Then    
    	\[|\mathcal{D}(\mathcal{F})| \leq \binom{n-t}{k-t}+\binom{n-t}{k-t-1}+\cdots+\binom{n-t}{0}.\]
    \end{theorem}


    \section{Preliminaries}
    
    To prove Theorem \ref{main-theorem} we require the following notions. 
     
    For a family $\mathcal{F} \subset 2^{[n]}$ and a non-negative integer $\ell$ let 
    \[\mathcal{D}^{(\ell)}(\mathcal{F}) 
    = \left\{ D \in \binom{[n]}{\ell}: \exists F, G \in \mathcal{F}, F\setminus G = D \right\}.\]
    Note that 
    \begin{equation*}
    	|\mathcal{D}(\mathcal{F})|
    	= \sum_{\ell\geq 0} |\mathcal{D}^{(\ell)}(\mathcal{F})|.
    \end{equation*}
    
    We call a family $\mathcal{F}$ to be a \textit{$t$-star} if there is some $t$-subset contained in every member of $\mathcal{F}$. Similarly, we call $\mathcal{F}$ to be \textit{the full $t$-star} all elements of $\binom{[n]}{k}$ containing a fix $t$-subset are the only elements of the family $\mathcal{F}$.
       
    \begin{lemma} \label{main-lemma}
    	Suppose that the family $\mathcal{F} \subset \binom{[n]}{k}$ is $t$-intersecting and that $\mathcal{D}^{(k-t)}(\mathcal{F})$ has $k+t+2$ members that are pairwise disjoint. Then $\mathcal{F}$ is a t-star. 
    \end{lemma} 

    \begin{proof}
    	We can assume that $k\geq t+1$, as the case $k = t$ is trivial. Let $D_0, D_1, \ldots, D_{k+t+1}$ be members of $\mathcal{D}^{(k-t)}(\mathcal{F})$ that are pairwise disjoint. As $D_i=F_i\setminus F_i^\prime$ for some $F_i, F_i^\prime \in \mathcal{F}$ and $\mathcal{F}$ is $t$-intersecting, for each $D_i$ there exist some subset $X_i$ of $2^{[n]}$ with $|X_i|=t$ and $D_i\cup X_i \in \mathcal{F}$. Observe that each $x \in X_0$ can contain in at most one $D_i$, without loss of generality let $D_{k+2}, D_{k+3}, \ldots, D_{k+t+1}$ be such $D_i$'s. Therefore, $(D_0 \cup X_0) \cap D_i = \emptyset$ for all $i=1,2,\ldots,k+1$.
    	
    	However, as $D_i \cup X_i \in \mathcal{F}$ and $\mathcal{F}$ is a $t$-intersecting family, we have in particular that $|(D_0 \cup X_0) \cap (D_i \cup X_i)| \geq t$. But, $D_i$'s are pairwise disjoint and $X_0$ has nothing common with $D_i$ for $1 \leq i \leq k+1$, the only possibility is that $X_i \subset (D_0\cup X_0)$ for $1 \leq i \leq k+1$. 
    	
        Further, as $|(D_1 \cup X_1) \cap (D_i \cup X_i)| \geq t$ for $1 \leq i \leq k+1$, we have $X_1 \cap D_i = \emptyset$, otherwise we would have $D_1 \cap D_i \neq \emptyset$ or $(D_0 \cup X_0) \cap D_1 \neq \emptyset$ giving contradictions. By similar argument $X_i \cap D_1 = \emptyset$. Thus, $X_1=X_i$ for $1 \leq i \leq k+1$. 
        
        Our claim is that $X_1 \in F$ for all $F \in \mathcal{F}$. Suppose that $X_1 \nsubseteq F$ for some $F \in \mathcal{F}$. Then $|(D_i \cup X_i) \cap F| \geq t$ implies $|D_i \cap F| \neq \emptyset$ for all $1 \leq i \leq k+1$, which further implies $|\mathcal{F}| \geq k+1$, a contradiction. 
    \end{proof}

    If $\mathcal{F} \subset \binom{[n]}{k}$ is a $t$-star with $X \subset F$ for all $F \in \mathcal{F}$, then $\mathcal{D}^{(\ell)}(\mathcal{F}) \subset \binom{[n]\setminus X}{\ell}$ for all $0 \leq \ell \leq k-t$. Thus, for $t$-stars 	$|\mathcal{D}(\mathcal{F})|
    \leq \sum_{\ell=0}^{k-t} |\mathcal{D}^{(\ell)}(\mathcal{F})|$. The equality holds when $\mathcal{F}$ is the full $t$-star.

    \begin{corollary} \label{main-corollary} 
    	It is sufficient to prove Theorem 4 for families $\mathcal{F}$ with $\mathcal{D}^{(k-t)}(\mathcal{F})$ not containing $k+t+2$ pairwise disjoint members. 
    \end{corollary}

     \section{The proof of Theorem \ref{main-theorem}}
     
    \begin{proof}[Proof of Theorem \ref{main-theorem}]
    	Define $f(k,t) = (k-t)(k+t+3)\lfloor\left( \frac{k+t+3}{k+t+1}\right)^{t-1}\rfloor$. Choose a family $\mathcal{D} \subset \binom{[n]}{k-t}$ such that it has $(k+t+3)\lfloor\left( \frac{k+t+3}{k+t+1}\right)^{t-1}\rfloor$ members that are pairwise disjoint. The selection of $\mathcal{D}$ has done in an uniform way. Then the expected value 
    	\[\mathbb{E}(|\mathcal{D} \cap \mathcal{D}^{(k-t)}(\mathcal{F})|) = (k+t+3)\left\lfloor\left( \frac{k+t+3}{k+t+1}\right)^{t-1}\right \rfloor \frac{1}{\binom{n}{k-t}} |\mathcal{D}^{(k-t)}(\mathcal{F})|.\]
    	But by Corollary \ref{main-corollary} we have
    	\[(k+t+3)\left\lfloor\left( \frac{k+t+3}{k+t+1}\right)^{t-1}\right \rfloor \frac{1}{\binom{n}{k-t}} |\mathcal{D}^{(k-t)}(\mathcal{F})| \leq k+t+1.\]
    	Thus,     	
    	\begin{align*}
    	|\mathcal{D}^{(k-t)}(\mathcal{F})| 
    	&\leq \frac{(k+t+1)^t}{(k+t+3)^t}\binom{n}{k-t}\\
    	&= \frac{(k+t+1)^t}{(k+t+3)^t} \cdot  \frac{n\cdot(n-1)\cdots(n-t+1)}{(n-k+t)\cdot(n-k+t-1)\cdots(n-k+1)} \cdot \binom{n-t}{k-t}\\
    	&< \frac{(k+t+1)^t}{(k+t+3)^t} \cdot  \left(\frac{n}{n-k+1}\right)^t \cdot \binom{n-t}{k-t}.  
    	\end{align*}
    As the function $(\frac{n}{n-k+1})^t$ is decreasing on $n$ we see that $(\frac{n}{n-k+1})^t \leq (\frac{f(k,t)}{f(k,t)-k+1})^t \leq (\frac{k+t+2}{k+t+1})^t$. Therefore we have 
    \begin{align}\label{D(k-1)-equation}
    	|\mathcal{D}^{(k-t)}(\mathcal{F})| 
    	&< \left(\frac{k+t+2}{k+t+3}\right)^t
    	\binom{n-t}{k-t} \nonumber\\
    	&< \frac{k+t+2}{k+t+3} \binom{n-t}{k-t} \nonumber\\
    	&= \left(1- \frac{1}{k+t+3}\right) \binom{n-t}{k-t}.
    \end{align}
    For all other values of $\ell=0,1,\ldots,k-t-1$, as $\mathcal{F} \subset \binom{[n]}{k}$ is $t$-intersecting, we have 
    \begin{equation}\label{D(0)-equation}
    	|\mathcal{D}^{(0)}(\mathcal{F})|
    	\leq 1 = \binom{n-t}{0},
    \end{equation}
    and
    \begin{equation}\label{D(l)-equation}
    |\mathcal{D}^{(\ell)}(\mathcal{F})|
    \leq \binom{n-t+1}{\ell}
    = \binom{n-t}{\ell}+\binom{n-t}{\ell-1}
    \end{equation}
    for $\ell=1,\ldots,k-t-1$. Observe that
    \begin{equation}\label{sum-equation}
    \sum_{\ell=0}^{k-t-2} \binom{n-t}{\ell} 
    \leq (k-t-1)\binom{n-t}{k-t-2}
    < \frac{1}{k+t+3} \binom{n-t}{k-t},
    \end{equation}
    where the latter inequality can be easily verified by just expanding the binomial coefficients. Hence, combining the equations (\ref{D(k-1)-equation}), (\ref{D(0)-equation}), (\ref{D(l)-equation}), and (\ref{sum-equation}) we obtain 
    \begin{align*}
    |\mathcal{D}(\mathcal{F})|
    &= \sum_{\ell=0}^{k-t} |\mathcal{D}^{(\ell)}(\mathcal{F})|\\
    &< \sum_{\ell=0}^{k-t-1} |\mathcal{D}^{(\ell)}(\mathcal{F})| + \left(1- \frac{1}{k+t+3}\right) \binom{n-t}{k-t}\\
    &= \sum_{\ell=0}^{k-t-1} \binom{n-t}{\ell}
    + \sum_{\ell=0}^{k-t-2} \binom{n-t}{\ell}
    + \left(1-\frac{1}{k+t+3}\right) \binom{n-t}{k-t}\\
    &< \sum_{\ell=0}^{k-t-1} \binom{n-t}{\ell}
    + \frac{1}{k+t+3}\binom{n-t}{k-t}
    + \left(1- \frac{1}{k+t+3}\right) \binom{n-t}{k-t}\\
    &= \sum_{\ell=0}^{k-t} \binom{n-t}{\ell}.
    \end{align*}
    This completes the proof of the theorem. 
    \end{proof}
    
    \section*{Acknowledgment}
    Both the authors wish to thank the Institute of Mathematical Sciences, Chennai for the financial support received through the institute postdoctoral program.  

	\bibliographystyle{amsplain}

\begin{thebibliography}{10}
		
		\bibitem{erdos} P. Erd\H{o}s, C. Ko, and R. Rado, Intersection theorems for systems of finite sets, Q. J. Math. 12 (1961), 310--320.
		
		\bibitem{frankl21} P. Frankl, On the number of distinct differences in an intersecting family, Disc. Math. 344 (2) (2021), Paper No. 112210.
		
		\bibitem{frankl78} P. Frankl, The Erd\H{o}s–Ko–Rado theorem is true for $n=ckt$, Combinatorics (Proc. Fifth Hungarian Colloq., Keszthey, 1976), Vol. I, 365--375, Colloq. Math. Soc. J\'{a}nos Bolyai, 18, North–Holland, 1978.
		
		\bibitem{frankl22} P. Frankl, S. Kiselev, and A. Kupavskii, Best possible bounds on the number of distinct differences in intersecting families, European J. Combin. 107 (2023), Paper No. 103601.	
		
		\bibitem{marica} J. Marica and J. Sch\"{o}nheim, Differences of sets and a problem of graham, Canad. Math. Bull. 12 (1969), 635--637.
		
		\bibitem{wilson} R. M. Wilson, The exact bound in the Erd\H{o}s–Ko–Rado theorem, Combinatorica, 4 (1984), 247--257.
		
	\end{thebibliography}

\end{document}